\documentclass{amsart}
\usepackage{graphicx} 
\usepackage{amsfonts}
\usepackage{amssymb}
\usepackage{amscd}
\usepackage[latin1]{inputenc}
\usepackage{indentfirst}
\usepackage{xcolor}
\usepackage{eufrak}
\usepackage{tikz}
\usepackage[normalem]{ulem}

\newtheorem{theorem}{Theorem}
\newtheorem{lemma}[theorem]{Lemma}
\newtheorem{remark}[theorem]{Remark}
\newtheorem{corollary}[theorem]{Corollary}
\newtheorem{proposition}[theorem]{Proposition}
\newtheorem{example}[theorem]{Example}

\title{Algebraic Lattices Arising from Congruence Submodules in Subfields of $p$-th Cyclotomic Fields}

\author{T. P. N. Neto}
\address{S\~ao Paulo State University, S\~ao Jos\'e do Rio Preto-SP, Brazil}
\email{trajano.nobrega@unesp.br}

\author{A. A. de Andrade}
\address{S\~ao Paulo State University, S\~ao Jos\'e do Rio Preto-SP, Brazil}
\email{antonio.andrade@unesp.br}

\author{J. L. R. Bastos}
\address{S\~ao Paulo State University, S\~ao Jos\'e do Rio Preto-SP, Brazil}
\email{jefferson.bastos@unesp.br}

\author{R. R. de Araujo}
\address{Federal Institute of S\~ao Paulo, Catanduva-SP, Brazil}
\email{robson.ricardo@ifsp.edu.br}

\author{J. C. Interlando}
\address{San Diego State University}
\email{interlan@sdsu.edu}

\begin{document}

\maketitle

\begin{abstract} The classical sphere packing problem, which remains unsolved, consists of determining how densely a large number of identical spheres can be packed together. In some sphere packings, the centers of the spheres in a sphere packing form a Euclidean lattice, which is a discrete additive subgroup of $\mathbb{R}^n$. Free $\mathbb{Z}$-modules in the ring of integers of an algebraic number field yield algebraic lattices via the canonical embedding. In this work, we present new constructions of algebraic lattices from certain families of $\mathbb{Z}$-modules in the ring of algebraic integers of subfields of the $p$-th cyclotomic field, where $p$ is a prime number. Within this framework, we compute lower bounds for the center density of these algebraic lattices and construct algebraic lattices having the best known packing density in dimensions 2, 3, and 5.
\end{abstract}

\section{Introduction} \label{sec-1}

The classical sphere packing problem seeks to determine the maximum density with which identical non-overlapping $n$-dimensional spheres can be packed in Euclidean space $\mathbb{R}^n$. Unlike space-filling solids such as hypercubes, spherical geometries inherently leave unoccupied space between adjacent spheres. A fundamental approach to constructing dense packings relies on Euclidean lattices, where sphere centers form a discrete additive subgroup of $\mathbb{R}^n$. In this context, rings of integers of algebraic number fields offer a structured algebraic framework to construct lattices via canonical embeddings. Extending this line of research, the primary objective of this work is to present new explicit constructions of algebraic lattices originating from congruence submodules in subfields of $p$-th cyclotomic fields, establish lower bounds for their center density, and recover optimal packing configurations in low dimensions.

A way to provide a sphere packing is using lattices. A (Euclidean) lattice is the set given by all $\mathbb{Z}$-linear combinations of some linearly independent vectors fixed in the space. The packing density of a lattice is the proportion of the $\mathbb{R}$-vector space generated by the lattice which is covered by the non-overlapping spheres of maximum radius centered at the points of lattice. The densest possible lattice packings have only been determined in dimensions $1$, $2$, $3$, $8$, and $24$ \cite{sloane,maryna}.

In telecommunications, usually the problem of finding good signal constellations for a Gaussian channel is associated with the search for lattices with high packing density. Algebraic number theory has been useful to the development of the theory of Euclidian lattices and has become of great interest for designing dense signal constellations well suited for transmission over AWGN channel. Furthermore, algebraic number theory has proven to be a useful mathematical tool in the search for lattices which can be the support for designing codes for both Gaussian and Rayleigh channels \cite{andrade-1,oggier-1,boutros,bottcher,oggier-2}.

Let $B=\{u_1,u_2,\ldots,u_n\}$ be a set with $n$ linearly independent vectors in $\mathbb{R}^n$. The set of all $\mathbb{Z}$-linear combinations of $B$ is defined to be a (full-rank) lattice with basis $B$. The matrix $G=\left[u_i\cdot u_j\right]_{n\times n}$ is called a Gram matrix of $\Lambda$. The determinant of $G$ is called the determinant of $\Lambda$ and its square root is the volume of $\Lambda$, denoted by $\textrm{vol}(\Lambda)$. The packing radius of $\Lambda$ is computed by $\rho:=(1/2)\min\{\|u\|~:~u\in\Lambda\setminus\{0\}\}$. The center density of $\Lambda$ is defined to be $\delta(\Lambda)=\rho^n/\textrm{vol}(\Lambda)$. The higher the center density, the denser the lattice packing realized by $\Lambda$, which shows the importance of this measure. For more details on lattices, see, e.g., \cite{costa,sloane}.



Algebraic number theory, particularly through cyclotomic fields and their subfields, has provided a rich framework for constructing lattices with high packing density \cite{jorge-1,oggier-1,bayer-1}. Building upon computational and practical constructions of lattices from cyclic extensions presented in \cite{interlando-1}, the primary contribution and novelty of this work rely on establishing a complete theoretical foundation and analytical generalization for these structures. Specifically, we present explicit constructions for two parametric families of $\mathbb{Z}$-modules, denoted by $\mathcal{M}_p$ and $\tilde{\mathcal{M}}_p$, within the ring of algebraic integers $\mathcal{O}_{\mathbb{K}}$ of a subfield $\mathbb{K} \subseteq \mathbb{Q}(\zeta_p)$ of degree $s$. Our main contribution is the systematic derivation of these submodules using algebraic orthogonality relations in the finite vector space $\mathbb{F}_p^s$ alongside suitable primitive roots modulo $p$. By rigorously analyzing the trace form restricted to these submodules, we compute their exact algebraic indices $[\mathcal{O}_{\mathbb{K}} : \mathcal{M}_p]$ and $[\mathcal{O}_{\mathbb{K}} : \tilde{\mathcal{M}}_p]$, prove closed-form lower bounds for the center density of the resulting algebraic lattices across generic dimensions $s$, and explain the theoretical mechanism by which these structured module constructions systematically recover the best known packing densities in low dimensions, specifically in dimensions $2, 3$, and $5$.

This work is organized as follows. In Section 2, we characterize the subfields of $\mathbb{Q}(\zeta_p)$ and explicitly describe their rings of algebraic integers. In Section 3, we present the trace form associated with these rings, which is used to compute the center density of the corresponding algebraic lattices. In Section 4, we summarize some basic facts about general algebraic lattices. Finally, in Section 5, we provide the constructions of algebraic lattices coming from the proposed $\mathbb{Z}$-modules $\mathcal{M}_p$ and $\tilde{\mathcal{M}}_p$, compute their indices, establish center density lower bounds, and highlight the optimal cases.


\section{Subfields of the $p$-th cyclotomic field} \label{sec-2}

Let $\mathbb{L}=\mathbb{Q}(\zeta_p)$ be the $p$-th cyclotomic field, where $p$ is a prime number. It is well known that, for every divisor $s$ of $p-1$, there exists a unique subfield $\mathbb{K}\subseteq \mathbb{L}$
with $[\mathbb{K}:\mathbb{Q}]=s$, by the Galois correspondence. Moreover, the rings of algebraic integers of $\mathbb{L}$ and of its maximal real subfield $\mathbb{Q}(\zeta_p+\zeta_p^{-1})$ are
\[ \mathcal{O}_{\mathbb{L}}=\mathbb{Z}[\zeta_p] \quad\text{and}\quad \mathcal{O}_{\mathbb{Q}(\zeta_p+\zeta_p^{-1})} =\mathbb{Z}[\zeta_p+\zeta_p^{-1}], \]
respectively \cite{wash}. Consequently, \[ \{1,\zeta_p,\zeta_p^2,\ldots,\zeta_p^{p-2}\} \ \ \mbox{and} \ \ \{1,\zeta_p+\zeta_p^{-1},\ldots,\zeta_p^{\frac{p-3}{2}}+\zeta_p^{-\frac{p-3}{2}}\} \]
are integral bases of $\mathbb{L}$ and $\mathbb{Q}(\zeta_p+\zeta_p^{-1})$, respectively.

Let $\mathbb{K}$ be an arbitrary subfield of $\mathbb{L}$, and write \[ s=[\mathbb{K}:\mathbb{Q}] \ \ \mbox{and} \ \ r=[\mathbb{L}:\mathbb{K}]=\frac{p-1}{s}. \] Then both $r$ and $s$ divide $p-1$. The Galois group
$G=\operatorname{Gal}(\mathbb{L}/\mathbb{Q})$ is canonically isomorphic to the multiplicative group $\mathbb{Z}_p^{\star} =(\mathbb{Z}/p\mathbb{Z})^{\star}$,
which is cyclic of order $p-1$. Let $\bar{u}$ be a generator of $\mathbb{Z}_p^{\star}$, and let $\sigma_u$ denote the automorphism of $\mathbb{L}$ determined by $\sigma_u(\zeta_p)=\zeta_p^{\,u}$.
Then \[ G=\langle\sigma_u\rangle = \{\sigma_u^0,\sigma_u,\sigma_u^2,\ldots,\sigma_u^{p-2}\}. \] Since every subgroup of a cyclic group is cyclic, both extensions $\mathbb{Q}\subseteq\mathbb{K}$ and $\mathbb{K}\subseteq\mathbb{L}$ are cyclic Galois extensions.

\begin{proposition} If $H=\operatorname{Gal}(\mathbb{L}/\mathbb{K}) = \langle \sigma_{u}^s\rangle$, then $\mathbb{K}=\mathbb{Q}(t)$, where
\[ t = \operatorname{Tr}_{\mathbb{L}/\mathbb{K}}(\zeta_p) = \sum_{j=0}^{r-1}\zeta_p^{\,u^{js}}. \]
\end{proposition}
\begin{proof} Since $H=\operatorname{Gal}(\mathbb{L}/\mathbb{K})$, the trace \[ t = \operatorname{Tr}_{\mathbb{L}/\mathbb{K}}(\zeta_p) = \sum_{\tau\in H}\tau(\zeta_p) \]
is fixed by every element of $H$. Hence $t\in\mathbb{L}^H=\mathbb{K}$. Let $ G=\operatorname{Gal}(\mathbb{L}/\mathbb{Q}) = \langle\sigma_u\rangle$, where $\sigma_u(\zeta_p)=\zeta_p^{\,u}$. Since
$H=\langle\sigma_{u}^s\rangle$,  its elements are \[ H = \{\sigma_u^0,\sigma_{u}^s,\sigma_{u}^{2s},\ldots,\sigma_{u}^{(r-1)s}\}. \] Therefore, $t = \zeta_p+\zeta_p^{u^s}
+\zeta_p^{u^{2s}} +\cdots+\zeta_p^{u^{(r-1)s}}$. For each integer $i$, with $0\le i\le s-1$, \[ \sigma_u^{\,i}(t) = \sum_{j=0}^{r-1} \zeta_p^{\,u^{\,i+js}}. \]
The exponents $u^{i},u^{i+s},\ldots,u^{i+(r-1)s}$ are precisely the elements of the coset $u^{i}H$ of $H$ in $G$. Since the cosets $H,uH,\ldots,u^{s-1}H$ are distinct, the conjugates
$t,\sigma_u(t),\ldots,\sigma_u^{\,s-1}(t)$ are pairwise distinct. Hence the orbit of $t$ under $G$ has cardinality $s$, so that
$[\mathbb{Q}(t):\mathbb{Q}] =s$. Since $t\in\mathbb{K}$, it follows that $\mathbb{Q}(t)\subseteq\mathbb{K}$. Finally, $
[\mathbb{Q}(t):\mathbb{Q}] = [\mathbb{K}:\mathbb{Q}] = s$, which implies $\mathbb{Q}(t)=\mathbb{K}$. Thus 
\[ t = \operatorname{Tr}_{\mathbb{L}/\mathbb{K}}(\zeta_p) = \sum_{j=0}^{r-1}\zeta_p^{\,u^{js}} \] is a primitive generator of the field $\mathbb{K}$.
\end{proof}

\begin{proposition} \label{lem-1} If $\zeta_p^{\,u^{i_1s+j_1}} = \zeta_p^{\,u^{i_2s+j_2}}$, where $1\le i_1,i_2\le r$ and $1\le j_1,j_2\le s$, then $i_1=i_2$ and $j_1=j_2$. \end{proposition}
\begin{proof} For hypothesis, $\zeta_p^{\,u^{i_1s+j_1}} = \zeta_p^{\,u^{i_2s+j_2}}\iff u^{i_1s+j_1} \equiv u^{i_2s+j_2} \pmod p \iff i_1s+j_1 \equiv i_2s+j_2 \pmod{p-1}$. Thus, $(i_1-i_2)s=(j_2-j_1)+rst$, for some $t\in\mathbb{Z}$. Therefore, $j_2\equiv j_1 \pmod s$, that is, $j_1=j_2$. So, $(i_1-i_2)s=rst$ and $i_1\equiv i_2 \pmod r$. Hence, $i_1=i_2$. \end{proof}

\begin{proposition}\label{prop-1} The set $\left\{ \zeta_p^{\,u^{is+j}}:\; 0\le i\le r-1,\; 1\le j\le s \right\}$  is linearly independent over $\mathbb{Q}$. \end{proposition}
\begin{proof} Suppose that \[ \sum_{j=1}^{s}\sum_{i=0}^{r-1} a_{ji}\,\zeta_p^{\,u^{is+j}} = 0, \] where $a_{ji}\in\mathbb{Q}$.
By Proposition \ref{lem-1}, the exponents $u^{is+j}$, with $0\le i\le r-1$ and $1\le j\le s$, are pairwise distinct modulo $p$. Thus, there are $rs=p-1$
such exponents. Hence this relation can be rewritten as \[ \sum_{k=1}^{p-1} b_k\zeta_p^{\,k} = 0, \] where each $b_k\in\mathbb{Q}$.
Since $\{\zeta_p,\zeta_p^2,\ldots,\zeta_p^{p-1}\}$ is a $\mathbb{Q}$-basis of $\mathbb{Q}(\zeta_p)$, it follows that $b_k=0$, for $1\le k\le p-1$. 
Therefore, \[ a_{ji}=0, \ \ \mbox{for} \ \ 0\le i\le r-1 \ \ \mbox{and} \ \ \;1\le j\le s), \] which proves that the given set is linearly independent over $\mathbb{Q}$. \end{proof}

Denote $
t=\operatorname{Tr}_{\mathbb{L}/\mathbb{K}}(\zeta_p)$. Since $t$ is fixed by every element of $H=\operatorname{Gal}(\mathbb{L}/\mathbb{K})$, it follows that $ t\in\mathbb{K}$, and therefore,
\[ \mathbb{Q}(t)\subseteq\mathbb{K}. \] Moreover, $t$ is an algebraic integer because it is the trace of the algebraic
integer $\zeta_p$. Hence $\operatorname{Tr}_{\mathbb{Q}(t)/\mathbb{Q}}(t)\in\mathbb{Z}$. Using the transitivity of the trace, it follows that
\[  -1 = \operatorname{Tr}_{\mathbb{L}/\mathbb{Q}}(\zeta_p) = \operatorname{Tr}_{\mathbb{K}/\mathbb{Q}} \!\left( \operatorname{Tr}_{\mathbb{L}/\mathbb{K}}(\zeta_p)\right) = \operatorname{Tr}_{\mathbb{K}/\mathbb{Q}}(t) = [\mathbb{K}:\mathbb{Q}(t)] \operatorname{Tr}_{\mathbb{Q}(t)/\mathbb{Q}}(t). \]
Since the right-hand side is the product of the positive integer $[\mathbb{K}:\mathbb{Q}(t)]$ and an integer, it follows that $[\mathbb{K}:\mathbb{Q}(t)]=1$.
Therefore, $\mathbb{K}=\mathbb{Q}(t)$.

\begin{theorem}\label{teo-1} The set $S = \left\{t,\,\sigma_u(t),\,\sigma_u^{2}(t),\,\ldots,\,\sigma_u^{\,s-1}(t)\right\}$ is a $\mathbb{Q}$-basis of $\mathbb{K}$. \end{theorem}
\begin{proof} Since $\operatorname{Gal}(\mathbb{L}/\mathbb{Q}) = \langle\sigma_u\rangle$ is cyclic of order $p-1$, the subgroup $H=\operatorname{Gal}(\mathbb{L}/\mathbb{K})$
has order $r=(p-1)/s$ and is generated by $\sigma_{u^s}$. Thus $H = \left\{ \sigma_u^0,\sigma_{u}^s,\sigma_{u}^{2s},\ldots,\sigma_{u}^{(r-1)s}\right\}$. Consequently, \[ t = \operatorname{Tr}_{\mathbb{L}/\mathbb{K}}(\zeta_p) = \sum_{\tau\in H}\tau(\zeta_p) =
\zeta_p+\zeta_p^{u^s}+\zeta_p^{u^{2s}}+\cdots+\zeta_p^{u^{(r-1)s}}. \] Since $\mathbb{K}/\mathbb{Q}$ is Galois, every automorphism $\sigma_u^i$ restricts to an automorphism of $\mathbb{K}$. Therefore,
$\sigma_u^i(t)\in\mathbb{K}$, for $0\le i\le s-1$. Moreover, \[ \sigma_u^i(t) = \sum_{j=0}^{r-1} \zeta_p^{\,u^{\,js+i}}, \ \ \mbox{for} \ \ 0\le i\le s-1. \]
Suppose that \[ \sum_{i=0}^{s-1} a_i\sigma_u^i(t)=0, \ \ \mbox{for} \ \ a_i\in\mathbb{Q}. \] Then \[ \sum_{i=0}^{s-1} a_i \left(
\sum_{j=0}^{r-1} \zeta_p^{\,u^{\,js+i}} \right) = 0. \] By Proposition~\ref{prop-1}, the set \[ \left\{ \zeta_p^{\,u^{\,js+i}}: 0\le i\le s-1,\; 0\le j\le r-1\right\} \]
is linearly independent over $\mathbb{Q}$. Hence $a_i=0$, for $0\le i\le s-1$. Therefore, \[ S = \left\{ t,\sigma_u(t),\ldots,\sigma_u^{\,s-1}(t) \right\} \] is linearly independent over $\mathbb{Q}$.
Finally, $|S| = s = [\mathbb{K}:\mathbb{Q}]$, so $S$ is a $\mathbb{Q}$-basis of $\mathbb{K}$. \end{proof}

\begin{corollary} \label{cor-1} The set $S=\left\{t,\,\sigma_u(t),\,\sigma_u^{2}(t),\,\ldots,\,\sigma_u^{\,s-1}(t)\right\}$ is a $\mathbb{Z}$-basis of $\mathcal{O}_{\mathbb{K}}$. \end{corollary}
\begin{proof} By Theorem~\ref{teo-1}, the set \[ S = \left\{ t,\,\sigma_u(t),\,\ldots,\,\sigma_u^{\,s-1}(t)\right\} \] is a $\mathbb{Q}$-basis of $\mathbb{K}$.
First, we prove that $\langle S\rangle_{\mathbb{Z}}\subseteq\mathcal{O}_{\mathbb{K}}$. Let \[ \alpha = \sum_{i=0}^{s-1}a_i\sigma_u^i(t), \ \ \mbox{with} \ \  a_i\in\mathbb{Z}. \]
Since \[ t = \sum_{j=0}^{r-1} \zeta_p^{\,u^{js}}, \] it follows that \[ \sigma_u^i(t) = \sum_{j=0}^{r-1} \zeta_p^{\,u^{js+i}}, \ \ \mbox{for} \ \ 0\le i\le s-1. \]
Each $\sigma_u^i(t)$ is a sum of roots of unity and therefore belongs to
$\mathbb{Z}[\zeta_p]=\mathcal{O}_{\mathbb{L}}$. Since $\sigma_u^i(t)\in\mathbb{K}$ and $\mathcal{O}_{\mathbb{K}} = \mathcal{O}_{\mathbb{L}}\cap\mathbb{K}$, it follows that
\[ \sigma_u^i(t)\in\mathcal{O}_{\mathbb{K}}. \] Hence $\alpha\in\mathcal{O}_{\mathbb{K}}$, and consequently $\langle S\rangle_{\mathbb{Z}} \subseteq \mathcal{O}_{\mathbb{K}}$.
Conversely, let $\alpha\in\mathcal{O}_{\mathbb{K}}$. Since $S$ is a $\mathbb{Q}$-basis of $\mathbb{K}$, there exist unique
coefficients $a_0,\ldots,a_{s-1}\in\mathbb{Q}$ such that \[ \alpha = \sum_{i=0}^{s-1} a_i\sigma_u^i(t). \] Using the above expression for the conjugates of $t$, it follows that 
\[ \alpha = \sum_{i=0}^{s-1} a_i \left( \sum_{j=0}^{r-1} \zeta_p^{\,u^{js+i}} \right). \] By Proposition \ref{lem-1}, the exponents $u^{js+i}$, for $0\le i\le s-1$ and $0\le j\le r-1$,
are pairwise distinct modulo $p$. Since there are $rs=p-1$ such exponents, they are precisely the nonzero residue classes modulo $p$. Hence
\[ \alpha = \sum_{k=1}^{p-1} b_k\zeta_p^k, \] where each coefficient $b_k$ is one of the numbers $a_0,\ldots,a_{s-1}$. Now, \[
\alpha\in\mathcal{O}_{\mathbb{K}} \subseteq \mathcal{O}_{\mathbb{L}} = \mathbb{Z}[\zeta_p]. \] Since $\{\zeta_p,\zeta_p^2,\ldots,\zeta_p^{p-1}\}$ is a $\mathbb{Z}$-basis of $\mathbb{Z}[\zeta_p]$, it follows that $b_k\in\mathbb{Z}$, for $1\le k\le p-1$. 
Therefore, $a_i\in\mathbb{Z}$, for $0\le i\le s-1$. Hence $\alpha\in\langle S\rangle_{\mathbb{Z}}$,  which proves that $\mathcal{O}_{\mathbb{K}} \subseteq\langle S\rangle_{\mathbb{Z}}$.
So $\mathcal{O}_{\mathbb{K}} = \langle S\rangle_{\mathbb{Z}}$, and therefore, \[ S= \left\{t,\,\sigma_u(t),\,\sigma_u^{2}(t),\, \ldots,\,\sigma_u^{\,s-1}(t) \right\} \] is a $\mathbb{Z}$-basis of $\mathcal{O}_{\mathbb{K}}$.
\end{proof}

\begin{remark} \label{rem-0} A cyclic number field $\mathbb{K}$ is said to admit an integral normal basis if $\{t,\theta(t),\theta^2(t),\ldots,\theta^{s-1}(t)\}$ is a $\mathbb{Z}$-basis for $\mathcal{O}_{\mathbb{K}}$, where $t\in \mathcal{O}_\mathbb{K}$ and $\theta$ is a generator of $\textrm{Gal}(\mathbb{K}/\mathbb{Q})$. Hence, from Corollary \ref{cor-1}, every subfield of $\mathbb{Q}(\zeta_p)$ admits an integral normal basis given by $\{t,\theta(t),\theta^2(t),\ldots,\theta^{s-1}(t)\}$, where $\theta=\sigma_{u}|_\mathbb{K}$. \end{remark}

\section{The trace form} \label{sec-3}

Let $\mathbb{L}=\mathbb{Q}(\zeta_p)$ be a $p$-th cyclotomic field, where $p$ is a prime number. Let $\mathbb{K}$ be a number field such that $\mathbb{K}\subseteq\mathbb{L}$, with $[\mathbb{K}:\mathbb{Q}]=s$ and $[\mathbb{L}:\mathbb{K}]=r$. Let $\theta$ be the generator of the Galois group $Gal(\mathbb{K}/\mathbb{Q})$. Let $\mathcal{O}_{\mathbb{L}}=\mathbb{Z}[\zeta_p]$ and $\mathcal{O}_{\mathbb{K}}$ be the ring of algebraic integers of $\mathbb{L}$ and $\mathbb{K}$, respectively.
In this section, we present an expression of the integral trace form, i.e., the trace of an element of the form $x\overline{x}$, where $x$ is an algebraic integer of $\mathcal{O}_{\mathbb{K}}$. The minimum of this trace form is an important parameter in the study of the packing density of the lattice realized by $\mathcal{O}_\mathbb{K}$ via the canonical embedding.

It is well-known that $$Tr_{\mathbb{L}}(\zeta_p^k)=\begin{cases} -1 & \textrm{if } k\not\equiv 0 \pmod p,\\ p-1 & \textrm{if } k\equiv 0\pmod p. \end{cases}$$ If $x=\displaystyle \sum_{i=0}^{p-2}a_i\zeta_p^i\in\mathcal{O}_\mathbb{L}=\mathbb{Z}[\zeta_p]$, with $a_0,a_1,\ldots,a_{p-2}\in\mathbb{Z}$, then
\[  \begin{array}{lll} x\overline{x} &=& (a_0+a_1\zeta_p+\cdots+a_{p-2}\zeta_p^{p-2})(a_0+a_1\zeta_p^{-1}+\cdots+a_{p-2}\zeta_p^{-(p-2)})\vspace{.2cm} \\ &=&(a_0^2+\cdots +a_{p-2}^2)+(a_0a_1+\cdots +a_{p-3}a_{p-2})(\zeta_p+\zeta_p^{-1})\vspace{.2cm}\\ && \displaystyle +\cdots + (a_0a_{p-3}+a_1a_{p-2})(\zeta_p^{p-3}+\zeta_p^{-(p-3)})+a_0a_{p-2}(\zeta_p^{p-2}+\zeta_p^{-(p-2)}).\end{array} \]
Denote $x_i=\zeta_p^{i}+\zeta_p^{-i}$ and $A_i=a_0a_i+a_1a_{i+1}+\cdots +a_{p-2-i}a_{p-2}$, for $i=1,\ldots,p-2$. With this notation, \begin{equation} \label{expxxbarra} x\overline{x}=A_0+A_1x_1+\cdots +A_{p-2}x_{p-2}. \end{equation}
Consequently, the trace form associated to $\mathbb{L}$ is given by the following:

\begin{theorem} \label{teo-2} If $x=\displaystyle\sum_{i=0}^{p-2}a_i\zeta_p^i\in\mathbb{Z}[\zeta_p]$, then \[ Tr_{\mathbb{L}}(x\overline{x})=\displaystyle p\sum_{i=0}^{p-2}a_i^2 - \left(\sum_{i=0}^{p-2}a_i\right)^2. \]\end{theorem} \begin{proof} 
From Equation (\ref{expxxbarra}), $x\overline{x}=A_0+A_1x_1+\cdots + A_{p-2}x_{p-2}$, where $A_i=a_0a_i+a_1a_{i+1}+\cdots +a_{p-2-i}a_{p-2}.$ Since
\[ Tr_{\mathbb{L}}(x_i)=Tr_{\mathbb{L}}(\zeta_p^i+\zeta_p^{-i}) = Tr_{\mathbb{L}}(\zeta_p^i)+Tr_{\mathbb{L}}(\zeta_p^{-i})=-1-1=-2\]
and the coefficient $x_0$ of $A_0$ is equal to $1$, we have that \[ Tr_{\mathbb{L}}(1)=[\mathbb{L}:\mathbb{Q}]\cdot 1 = p-1. \] Thus,
\[ \begin{array}{lll} Tr_{\mathbb{L}}(x\overline{x}) &=& (p-1)A_0-2(A_1+A_2+\cdots+A_{p-2}) =(p-1)A_0\vspace{.2cm}\\ && -2(a_0a_1+\cdots+a_{p-3}a_{p-2}+a_0a_2+\cdots +a_{p-4}a_{p-2}+\cdots\vspace{.2cm}\\
&& +a_0a_{p-3}+a_1a_{p-2}+a_0a_{p-2}).\end{array} \] Hence, \[ Tr_{\mathbb{L}}(x\overline{x})=p\displaystyle\sum_{i=0}^{p-2}a_i^2-\left(\displaystyle\sum_{i=0}^{p-2}a_i^2+2\displaystyle\sum_{0\leq
i< j\leq p-2}a_ia_j\right). \] Therefore, \begin{equation} \label{eqtra1} Tr_{\mathbb{L}}(x\overline{x})=p\displaystyle\sum_{i=0}^{p-2}a_i^2-\left(\displaystyle\sum_{i=0}^{p-2}a_i\right)^2,
\end{equation} which proves the result. \end{proof}

From Equation ($\ref{eqtra1}$), it follows that \[ \begin{array}{lll} \displaystyle p\sum_{i=0}^{p-2}a_i^2-\left(\sum_{i=0}^{p-2}a_i\right)^2 & = &
 \displaystyle (p-1)\sum_{i=0}^{p-2}a_i^2 - \displaystyle\sum_{o\leq i<j\leq p-2} a_ia_j \vspace{.2cm}\\ &=& \displaystyle\sum_{i=0}^{p-2}a_i^2
+(p-2)\sum_{i=0}^{p-2}a_i^2- \sum_{o\leq i<j\leq p-2}a_ia_j \vspace{.2cm} \\ &=&   \displaystyle \sum_{i=0}^{p-2}a_i^2 + \sum_{o\leq i<j\leq p-2}(a_i-a_j)^2. \end{array} \]
Therefore, \[ Tr_{\mathbb{L}}(x\overline{x})=\displaystyle\sum_{i=0}^{p-2}a_i^2+\displaystyle\sum_{0\leq i< j\leq p-2}(a_i-a_j)^2. \] Furthermore, from Equation (\ref{eqtra1}), \[ \min\{ Tr_{\mathbb{L}/\mathbb{Q}}(x\overline{x}); \  x \in \mathcal{O}_{\mathbb{L}}, \ x \neq 0\} = p-1, \] where the minimum is attained at $x=1$. 


\begin{theorem}\label{prop500} If $x=\displaystyle \sum_{i=0}^{s-1} a_i\theta^i(t) \in \mathcal{O}_{\mathbb{K}}$, where $a_0,a_1,\ldots,a_{s-1}\in\mathbb{Z}$ and $t=Tr_{\mathbb{L}/\mathbb{K}}(\zeta_p)$, then \[ Tr_{\mathbb{K}}(x\overline{x}) = p\sum_{i=0}^{s-1}a_i^2-\frac{p-1}{s}\left(\sum_{i=0}^{s-1}a_i\right)^2 . \] \end{theorem} 
\begin{proof} Let $x=\displaystyle \sum_{i=0}^{s-1} a_i\theta^i(t) \in \mathcal{O}_{\mathbb{K}}$, with $a_0,a_1,\ldots,a_{s-1}\in\mathbb{Z}$. Since $t=Tr_{\mathbb{L}/\mathbb{K}}(\zeta_p)\in \mathcal{O}_{\mathbb{K}}$, it follows that $t={\zeta_p}^{u^s}+{\zeta_p}^{u^{2s}}+\cdots+{\zeta_p}^{u^{rs}}$. So, $\sigma_{u^i}(t)={\zeta_p}^{u^{s+i}}+\cdots+{\zeta_p}^{u^{rs+i}}$, for $i=0,1,\ldots,s-1$, and
$$x = \sum_{i=0}^{s-1} a_i\left(\sum_{j=0}^{r-1}\zeta_p^{u^{js+i}} \right),$$ that is, $x=a_0({\zeta_p}^{u^s}+{\zeta_p}^{u^{2s}}+\cdots+{\zeta_p}^{u^{rs}})+a_1({\zeta_p}^{u^{s+1}}+{\zeta_p}^{u^{2s+1}}+\cdots+{\zeta_p}^{u^{rs+1}})+\cdots+a_{s-1}({\zeta_p}^{u^{s+s-1}}+\cdots+{\zeta_p}^{u^{rs+s-1}})$. 
From Theorem \ref{teo-2}, it follows that \[ Tr_{\mathbb{L}}(x\overline{x}) = r\sum_{i=0}^{s-1}a{_{i}^{2}} + {r^2}\sum_{0\leq i < j \leq s-1}^{}{(a_i - a_j)}^{2}. \] Since  $Tr_{\mathbb{L}/\mathbb{Q}}(x\overline{x}) = [\mathbb{L}:\mathbb{K}] Tr_{\mathbb{K}}(x\overline{x})$, then
$$Tr_{\mathbb{K}}(x\overline{x})=\frac{1}{r}Tr_{\mathbb{L}/\mathbb{Q}}(x\overline{x}),$$ and so, \[ Tr_{\mathbb{K}}(x\overline{x})=\sum_{i=0}^{s-1}a_{i}^{2} + r\sum_{0 \leq i < j \leq s-1}^{}{(a_i - a_j)}^{2}. \] Therefore, \[ Tr_{\mathbb{K}}(x\overline{x}) = p\sum_{i=0}^{s-1}a_i^2-\frac{p-1}{s} \left(\sum_{i=0}^{s-1}a_i\right)^2, \] which proves the result. \end{proof} 

\section{Algebraic lattices} \label{sec-4}

If $\mathbb{K}$ is an algebraic number field of degree $n$, there are exactly $n$ distinct $\mathbb{Q}$-monomorphisms $\sigma_i:\mathbb{K}\rightarrow\mathbb{C}$, for $i=1,2,\ldots,n$. A $\mathbb{Q}$-monomorphism $\sigma_i$ is said to be real if $\sigma_i(\mathbb{K})\subseteq\mathbb{R}$. Otherwise, $\sigma_i$ is said to be imaginary. It is a fact that $n=r_1+2r_2$, where $r_1$ is the number of real $\mathbb{Q}$-monomorphisms of $\mathbb{K}$ and $r_2$ is the number of pairs of imaginary $\mathbb{Q}$-monomorphisms. The number field $\mathbb{K}$ is said to be totally real if $\sigma_i$ is real for all $i=1,2,\ldots,n$. In turn, the number field $\mathbb{K}$ is totally imaginary if $\sigma_i$ is imaginary for all $i=1,2,\ldots,n$. 

Let $\sigma_1,\sigma_2,\ldots,\sigma_n$ denote the $\mathbb{Q}$-monomorphisms of $\mathbb{K}$ ordered such that $\sigma_1,\ldots,\sigma_{r_1}$ are the real monomorphisms and $\sigma_{r_1+1},\ldots,\sigma_{r_1+r_2}$ and $\sigma_{r_1+r_2+j}=\overline{\sigma_{r_1+j}}$, for $j=1,2,\ldots,r_2$, are the imaginary monomorphisms. The trace of an element $x\in\mathbb{K}$ over $\mathbb{Q}$ can be computed by $Tr_{\mathbb{K}}(x)=\sum_{i=1}^n\sigma_i(x)$. The discriminant of $\mathbb{K}$ over $\mathbb{Q}$ is defined by $D(\mathbb{K})=\det(Tr_{\mathbb{K}}(\alpha_i\alpha_j))_{i,j=1}^n$, where $\{\alpha_1,\alpha_2,\ldots,\alpha_n\}$ is an integral basis of $\mathcal{O}_{\mathbb{K}}$ (for details, see, e.g., \cite{samuel}). 

The canonical embedding $\sigma:\mathbb{K}\rightarrow\mathbb{R}^{r_1}\times\mathbb{C}^{r_2}$ is defined by  \[ \sigma(x) = (\sigma_1(x),\ldots,\sigma_{r_1}(x),\sigma_{r_1+1}(x),\ldots,\sigma_{r_1+r_2}(x)), \] where $x\in\mathbb{K}$ \cite[Section 4.2]{samuel}. We shall frequently identify $\mathbb{R}^{r_1}\times\mathbb{C}^{r_2}$ with $\mathbb{R}^n$, and thus, we consider $\sigma:\mathbb{K}\rightarrow\mathbb{R}^{n}$ given, for each $x\in \mathbb{K}$, by 
\begin{multline*} \sigma(x) = (\sigma_1(x),\ldots,\sigma_{r_1}(x),\mathfrak{R}(\sigma_{r_1+1}(x)),\mathfrak{I}(\sigma_{r_1+1}(x)),\\  \ldots,\mathfrak{R}(\sigma_{r_1+r_2}(x)),\mathfrak{I}(\sigma_{r_1+r_2}(x))), \end{multline*}
where $\mathfrak{R}(\beta)$ and $\mathfrak{I}(\beta)$ denote the real and imaginary parts of the complex number $\beta$, respectively.

Suppose $\overline{\sigma_i(x)}=\sigma_i(\overline{x})$ for all $x\in\mathbb{K}$ and for all $i=0,1,\ldots,n$. If $\mathcal{M}$ is a free $\mathbb{Z}$-module of $\mathcal{O}_{\mathbb{K}}$ of rank $n$, then $\sigma(\mathcal{M})$ is an $n$-dimensional lattice whose minimum is given by $\min\{||\sigma(x)||^2:x\in\mathcal{M},\ x\neq 0\}$, where \[ ||\sigma(x)||^2 = \left\{ \begin{array}{ll} Tr_{\mathbb{K}}(x^2) & \mbox{if} \ \mathbb{K} \ \mbox{is totally real,}\\  \frac{1}{2}Tr_{\mathbb{K}}(x\overline{x}) & \mbox{if} \ \mathbb{K} \ \mbox{is totally complex}. \end{array} \right. \] 

The center density of the lattice $\sigma(\mathcal{M})$ is given by \begin{equation} \label{densidade} \delta(\sigma(\mathcal{M})) = \frac{\rho^{n/2}}{2^n\sqrt{|D(\mathbb{K})|}[\mathcal{O}_{\mathbb{K}}:\mathcal{M}]}, \end{equation} where $\rho$ denotes the minimum of trace form of $\mathbb{K}$ over $\mathcal{M}$ and $[\mathcal{O}_{\mathbb{K}}:\mathcal{M}]$ denotes the index of $\mathcal{M}$ in $\mathcal{O}_{\mathbb{K}}$ (see, e.g., \cite{samuel}). 
In particular, if $\mathbb{K}$ is a subfield of $\mathbb{L}=\mathbb{Q}(\zeta_p)$ of degree $s$, then $D(\mathbb{K})=\pm p^{s-1}$ (\cite[Corollary 4.2]{trajano}) and 
\begin{equation}\label{eq_delta} \delta(\sigma(\mathcal{M})) = \frac{\rho^{s/2}}{2^sp^{\frac{s-1}{2}}[\mathcal{O}_{\mathbb{K}}:\mathcal{M}]}. \end{equation}
Furthermore, the map \[ \begin{array}{llll} \varphi:&\mathbb{Z}^s & \rightarrow & \mathcal{O}_{\mathbb{K}}\\ & {\bf x}=(a_0,a_1,\ldots,a_{s-1}) & \mapsto & \displaystyle \varphi({\bf x})=\sum_{i=0}^{s-1}a_i\theta^i(t), \end{array} \] is an isomorphism of $\mathbb{Z}$-modules, where $\theta$ is a generator of $Gal(\mathbb{K}/\mathbb{Q})$. The map $\phi=Tr_{\mathbb{K}}\circ\varphi:\mathbb{Z}^s\rightarrow \mathbb{N}\cup\{0\}$ given by \[ \phi({\bf x})=p\sum_{i=0}^{s-1}a_i^2-\frac{p-1}{s}\left(\sum_{i=0}^{s-1}a_i\right)^2, \] for each ${\bf x}=(a_0,a_1,\ldots,a_{s-1})\in\mathbb{Z}^s$, is a positive-definite quadratic form.


\section{Constructions of algebraic lattices} \label{sec-5}

In this section, we provide a construction of algebraic lattices coming from certain modules in the ring of integers of $\mathbb{K}\subseteq\mathbb{Q}(\zeta_p)$. These modules are obtained from isomorphic subgroups corresponding to some special subspaces of $\mathbb{F}_p^s$, as presented below.

Consider ${\bf v}=(a_0,a_1,\ldots,a_{s-1})\in \mathbb{Z}^s$
and $\overline{\bf v}=(\overline{a_0},\overline{a_1},\ldots,\overline{a_{s-1}})$ in the vector space $V=\mathbb{F}_{p}^s$, where $\mathbb{F}_p$ is the finite field of $p$ elements, and where $\overline{a}\in\mathbb{F}_p$ denotes the reduction modulo $p$ of $a\in\mathbb{Z}$. Let \[ \mathcal{M}_{\bf v} = \{{\bf x}\in\mathbb{Z}^s:{\bf x}\cdot{\bf v}\equiv 0\pmod p\} = \mathcal{M}_{\bf v} = \{{\bf x}\in\mathbb{Z}^s:\overline{\bf x}\cdot\overline{\bf v}=\overline{0} \ \ \mbox{in} \ \ \mathbb{F}_p\}, \] where  $u\cdot v$ denotes the usual dot product between two vectors $u$ and $v$.

Since $\mathbb{F}_p^*$ is a cyclic group, there exists an element $\beta\in\mathbb{Z}$ such that $\beta$ is a $(p-1)$-th primitive root modulo $p$, that is, $\beta^{p-1}\equiv 1\pmod p$ and $\beta^j\not\equiv 1\pmod p$ for $j=1,2,\ldots,p-2$. The element $\alpha=\beta^{\frac{p-1}{s}}$ is a $s$-th primitive root modulo $p$. For each $i=0,1,\ldots,s-1$, denote \begin{equation} \label{eq-0} {\bf \mu_i}= \left(1,\alpha^i,\alpha^{2i},\ldots,\alpha^{(s-1)i}\right) \in \mathbb{Z}^s \end{equation} and \[ \overline{{\bf \mu_i}} = \left(1,\overline{\alpha}^i,\overline{\alpha}^{2i},\ldots,\overline{\alpha}^{(s-1)i}\right) \in \mathbb{F}_p^s. \] It is a known fact that the Vandermonde matrix \[ M := \left[ \begin{array}{cccc} 1 & 1 & \cdots & 1\\ 1 & \overline{\alpha} & \cdots & \overline{\alpha}^{(s-1)} \\ \vdots & \vdots & \ddots & \vdots \\ 1 & \overline{\alpha}^{s-1} & \cdots & (\overline{\alpha}^{s-1})^{s-1} \end{array} \right], \] has determinant
\[ \det(M) = \prod_{\substack{0\leq i<j<s}}(\overline{\alpha}^i-\overline{\alpha}^j), \] which is nonzero because
\begin{multline}\label{eq-1} \overline{\mu_i}\cdot\overline{\mu_j}  = \sum_{k=0}^{s-1} \overline{\alpha}^{k(i+j)} = \begin{cases} \frac{(\beta^{i+j})^s-1}{\beta^{i+j}-1} & \mbox{if } i+j\not\equiv 0\pmod s\\ s & \mbox{if } i+j\equiv 0\pmod s    \end{cases}\\ =  \begin{cases}
0 & \mbox{if } i+j\not\equiv 0\pmod s\\ s & \mbox{if } i+j\equiv 0\pmod s.\end{cases} \end{multline} Therefore, $\{\overline{\mu_0},\overline{\mu_1},\ldots,\overline{\mu_{s-1}}\}$ is a basis of $\mathbb{F}_p^s$.
If $s\equiv 1\pmod 2$, consider the subspace \[ W = \left\langle\overline{\mu_0},\overline{\mu_1},\ldots,\overline{\mu_{\frac{s-1}{2}}}\right\rangle_{\mathbb{F}_p} \subseteq\mathbb{F}_p^s. \]
Otherwise, if $s\equiv 0\pmod 2$, consider
\[  W=\left\langle\overline{\bf \mu_0},\overline{\bf \mu_1},\ldots,\overline{\bf \mu_{\frac{s}{2}}}\right\rangle_{\mathbb{F}_p} \subseteq\mathbb{F}_p^s. \] In each case, let \[ W^{\perp} = \{ {\bf x} \in \mathbb{F}_p^s:{\bf x}\cdot{\bf y}=0, \ \ \mbox{for all} \ \ {\bf y}\in W\} \] be the orthogonal complement of $W$.

\begin{proposition}\label{prop-2} If $s\equiv 1(mod\ 2)$, the orthogonal complement of $W$ is given by $$W^{\perp} = \left\langle\overline{\bf \mu_1},\overline{\bf \mu_2},\ldots,\overline{\bf \mu_{\frac{s-1}{2}}}\right\rangle_{\mathbb{F}_p} \subseteq\mathbb{F}_p^s.$$ \end{proposition}
\begin{proof} Let ${\bf x} = \displaystyle \sum_{i=0}^{s-1}a_i\overline{\bf \mu_i}\in \mathbb{F}_p^s$. Thus,  
\[ {\bf x}\in W^{\perp} \iff {\bf x}\cdot \overline{\bf \mu_j}=0, \ \ \mbox{for all} \ \ j=0,1,\ldots,\frac{s-1}{2}, \] that is, \[ \overline{x}\in W^{\perp} \iff \sum_{i=0}^{s-1} a_i\ \overline{\bf \mu_i}\cdot \overline{\bf\mu_j}=0, \ \ \mbox{for all} \ \ j=0,1,\ldots,\frac{s-1}{2}. \] Therefore, \[ {\bf x}\in W^{\perp}
\iff a_0s=a_{\frac{s+1}{2}}s = \ldots = a_{s-1}s =0, \] which implies that \[ a_i=0,~\mbox{for } i=0,a_{\frac{s+1}{2}},\ldots,a_{s-1}. \] So, ${\bf x}\in W^\perp$ if and only if ${\bf x}=a_1{\bf\mu_1}+\cdots+a_{\frac{s-1}{2}}{\bf \mu_{\frac{s-1}{2}}}$, which proves the result. \end{proof}

Similarly to Proposition \ref{prop-2}, we obtain the following result:

\begin{proposition}
    \label{cor-0} If $s\equiv 0\pmod 2$, the orthogonal complement of $W$ is given by $$W^{\perp} = \left<\overline{\bf \mu_1},\overline{\bf \mu_2},\ldots,\overline{\bf \mu_{\frac{s}{2}-1}}\right>.$$
\end{proposition} 

\subsection{The $\mathbb{Z}$-modules $\mathcal{M}_p$.}\label{subsection_Mp}

Since $\mathbb{K}$ is an algebraic number field of degree $s$, the $\mathbb{Z}$-modules $\mathcal{O}_\mathbb{K}$ and $\mathbb{Z}^s$ are isomorphic. More precisely, by the results of Section \ref{sec-2}, the map
\begin{equation}\label{eq_iso}
  \phi\left(a_0,a_1,\ldots,a_{s-1}\right)=\sum_{i=0}^{s-1} a_i\theta^i(t)  
\end{equation}
defines a $\mathbb{Z}$-module isomorphism from $\mathbb{Z}^s$ to $\mathcal{O}_\mathbb{K}$. Accordingly, for simplicity, we shall write ${\bf x}:=\phi({\bf x})\in\mathcal{O}_\mathbb{K}$ for all ${\bf x}=(a_0,a_1,\ldots,a_{s-1})\in\mathbb{Z}^s$.

Firstly, consider $s \equiv 1\pmod 2$. Let
\begin{equation}\label{M_p_s_1}
    \mathcal{M}_p:=\left\{{\bf x}\in\mathcal{O}_{\mathbb{K}}:{\bf x}\cdot {\bf \mu_i}\equiv 0\pmod p, \ \forall~i=0,1,\ldots,\frac{s-1}{2}\right\}
\end{equation}
be a $\mathbb{Z}$-submodule of $\mathcal{O}_{\mathbb{K}}$. As in the previous subsection, consider $$W = \left\langle\overline{\mu_0},\overline{\mu_1},\ldots,\overline{\mu_{\frac{s-1}{2}}}\right\rangle_{\mathbb{F}_p} \subseteq\mathbb{F}_p^s$$ and its orthogonal complement $W^{\perp} = \left\langle\overline{\bf \mu_1},\overline{\bf \mu_2},\ldots,\overline{\bf \mu_{\frac{s-1}{2}}}\right\rangle_{\mathbb{F}_p}$ (Proposition \ref{prop-2}).

If ${\bf x}\in\mathcal{M}_p$, then $\overline{\bf x}\in W^{\perp}$. Since $W^{\perp}\subseteq W$, it follows that $\overline{\bf x}\cdot\overline{\bf x}=\overline{0}$, that is, ${\bf x}\cdot {\bf x}\equiv 0\pmod p$. Furthermore, \[ {\bf x}\cdot \overline{\bf x} = \sum_{i=0}^{s-1}a_i^2 . \] If ${\bf x}\in W$, then ${\bf x}\cdot {\bf \mu_0}\equiv 0\pmod p$. Thus, \[ \sum_{i=0}^{s-1}a_i\equiv 0\pmod p, \ \ \mbox{implying} \ \ \left(\sum_{i=0}^{s-1}a_i\right)^2\equiv 0\pmod p^2. \] 
Since \[ Tr_{\mathbb{K}}({\bf x}\overline{\bf x}) = p\sum_{i=0}^{s-1}a_i^2 -\frac{p-1}{s}\left(\sum_{i=0}^{s-1}a_i\right)^2 ,\] it follows that \[ Tr_{\mathbb{K}}({\bf x}\overline{\bf x}) \equiv 0 \pmod {p^2}. \] If $t=\min\{Tr_{\mathbb{K}}({\bf x}\overline{\bf x}): {\bf x}\in\mathcal{M}_p, {\bf x}\neq 0\}$, then \[ \delta(\sigma(\mathcal{M}_p)) = \frac{t^{\frac{s}{2}}}{2^sp^\frac{s-1}{2}[\mathcal{O}_{\mathbb{K}}:\mathcal{M}_p]}. \] Now, if 
\[ \sum_{i=0}^{s-1}a_i\equiv 0\pmod 2, \ \ \mbox{then} \ \ \sum_{i=0}^{s-1}a_i^2\equiv 0 \pmod 2, \] and, therefore, \[ Tr_{\mathbb{K}}({\bf x}\overline{\bf x}) \equiv 0\pmod {2p^2}. \] Thus, \begin{equation} \label{eq-2} \delta(\sigma(\mathcal{M}_p)) \geq \frac{(2p^2)^{\frac{s}{2}}}{2^sp^{\frac{s-1}{2}}[\mathcal{O}_{\mathbb{K}}:\mathcal{M}_p]}. \end{equation}

Now, consider $s \equiv 0\pmod 2$. Let
\begin{equation} \label{M_p_s_0}
   \mathcal{M}_p := \{{\bf x}\in\mathcal{O}_{\mathbb{K}}:{\bf x}\cdot {\bf \mu}_i\equiv 0 \pmod{p}, \ \ \forall i=0,1,\ldots,s/2\} 
\end{equation}
\[  \] be a $\mathbb{Z}$-submodule of $\mathcal{O}_{\mathbb{K}}$.
As in the previous subsection, consider
\[  W=\left\langle\overline{\bf \mu_0},\overline{\bf \mu_1},\ldots,\overline{\bf \mu_{\frac{s}{2}}}\right\rangle_{\mathbb{F}_p} \subseteq\mathbb{F}_p^s \]
and its orthogonal complement $W^{\perp} = \{ {\bf x} \in \mathbb{F}_p^s:{\bf x}\cdot{\bf y}=0, \ \ \mbox{for all} \ \ {\bf y}\in W\}$.

If $\overline{\bf x}\in W^{\perp}$, then $\overline{\bf x}\cdot\overline{\bf x}=\overline{0}$, that is, $\displaystyle {\bf x}\cdot {\bf x}=\sum_{i=0}^{s-1}a_i^2\equiv 0 \pmod{p}$. Thus, \[ Tr_{\mathbb{K}}({\bf x}\cdot\overline{\bf x})\equiv 0 \pmod{p^2}. \] Therefore, \begin{equation}\label{eq-s-cong1}
    \delta(\sigma(\mathcal{M}_p)) \geq \frac{(p^2)^{\frac{s}{2}}}{2^sp^{\frac{s-1}{2}}[\mathcal{O}_{\mathbb{K}}:\mathcal{M}_p]}. \end{equation}

\subsection{The index $[\mathcal{O}_{\mathbb{K}}:\mathcal{M}_p]$}

Initially, in this subsection, we present some elementary results in algebra that are important for the subsequent computation of the index $[\mathcal{O}_{\mathbb{K}}:\mathcal{M}_p]$. Their proofs are omitted, since they follow from standard results in abstract algebra.

\begin{lemma} \label{lem-2} Let $N_1$ and $N_2$ be the submodules of a $\mathbb{Z}$-module $N$. Let $\pi:N_1\rightarrow \frac{N}{N_2}$ be a homomorphism given by $\pi(x)=x+N_2$, where $x\in N_1$. Then $\ker(\pi) = N_1\cap N_2$ and $N_1/(N_1\cap N_2)$ is isomorphic to $Im(\pi)\subseteq N/N_2$. Furthermore, if $[N:N_2]<\infty$, then $[N_1:N_1\cap N_2]<\infty$ is a divisor of $[N:N_2]$.\end{lemma}

\begin{lemma} \label{cor-2} For a prime number $p$, if $[N:N_1]=[N:N_2]=p$ and $N_1\neq N_2$, then $[N_1:N_1\cap N_2]=p$. \end{lemma} 

\begin{lemma} \label{lem-3} Let $N_1\subseteq N_2\subseteq N_3$ be $\mathbb{Z}$-modules. If $[N_3:N_2]$ and $[N_2:N_1]$ are finite, then $[N_3:N_1]=[N_3:N_2][N_2:N_1]$ is also finite. \end{lemma}

\begin{lemma} \label{lem-4} Let $a_1,a_2,\ldots,a_s\in\mathbb{Z}\setminus\{0\}$. If $d=\gcd(a_1,a_2,\ldots,a_s)$, then $d\mathbb{Z}=a_1\mathbb{Z}+\cdots+a_s\mathbb{Z}$. \end{lemma}

\begin{lemma} \label{cor-3} If $m\in\mathbb{Z}$ is a positive integer, then the map

\[\begin{array}{llll} \pi:&\mathbb{Z}^s&\rightarrow& \mathbb{Z}/m\mathbb{Z}\\
& (x_1,x_2,\ldots,x_s) &\mapsto & \overline{\sum_{i=1}^s x_i a_i}
\end{array} \]
is a homomorphism of $\mathbb{Z}$-modules. \end{lemma}

From Lemma \ref{cor-3} and the Isomorphism Theorem, it follows that \[ \mathbb{Z}^s/\ker(\pi) \simeq Im(\pi)\subseteq \mathbb{Z}/m\mathbb{Z}. \] Since $d=\gcd(a_1,a_2,\ldots,a_s)$, it follows from Lemma \ref{lem-4} that $\sum_{i=1}^s x_ia_i \in d\mathbb{Z}$. Thus, \[ Im(\pi) =\{\overline{0},\overline{d},\overline{2d},\ldots,\overline{(k-1)d}\}, \] where $kd$ is minimum positive integer divisible by $m$. Thus, $k=\frac{m}{gcd(m,d)}$. This implies that \[ [\mathbb{Z}^s:\ker(\pi)] = \frac{m}{gcd(m,d)}. \] If $\gcd(m,d)=1$, then $\pi$ is surjective and
\begin{multline}\label{eq-nova}
   \ker(\pi)=\left\{(x_1,x_2,\ldots,x_n)\in\mathbb{Z}^s:\displaystyle \sum_{i=1}^n x_ia_i\equiv 0 \pmod m\right\}\\
   = \{{\bf u}\in\mathbb{Z}^n:{\bf u}\cdot{\bf v}\equiv 0 \pmod m\},
\end{multline}
where ${\bf v}=(a_1,a_2,\ldots,a_n)\in\mathbb{Z}^n$. Hence, in this case, $\mathbb{Z}^n/\ker(\pi)$ is isomorphic to $\mathbb{Z}/m\mathbb{Z}$, which implies $[\mathbb{Z}^n:\ker(\pi)]=m$.
For each $i=0,1,\ldots,s-1$, consider ${\bf \mu_i}= \left(1,\alpha^i,\alpha^{2i},\ldots,\alpha^{(s-1)i}\right) \in \mathbb{Z}^s$ as in Equation \ref{eq-0}. Let \[ \mathcal{N}_i=\{{\bf x}\in\mathbb{Z}^s: {\bf x}\cdot{\bf \mu_i}\equiv 0 \pmod p\} \] be a $\mathbb{Z}$-submodule of $\mathbb{Z}^s$. Since $\gcd(1,\alpha^i,\alpha^{2i},\ldots,\alpha^{(s-1)i})$, it follows from (\ref{eq-nova}) that $$[\mathbb{Z}^s:\mathcal{N}_i]=p.$$

\begin{lemma} \label{lem-5} $\mathcal{N}_i\neq \mathcal{N}_j$ for each pair $i\neq j\in\{0,1,\ldots,s-1\}$. \end{lemma} \begin{proof} Suppose $i\neq j$, for $i,j=0,1,\ldots,s-1$. From (\ref{eq-1}), it follows that ${\bf \mu_{s-j}}\cdot{\bf \mu_j}=s$, thus ${\bf \mu_{s-j}}\notin \mathcal{N}_j$ for all $j$. Moreover, ${\bf \mu_{s-j}}\cdot{\bf \mu_i}=0$, which implies ${\bf \mu_{s-j}}\in\mathcal{N}_i$. Therefore, $\mathcal{N}_i\not\subseteq \mathcal{N}_j$. Analogously, $\mathcal{N}_j\not\subseteq \mathcal{N}_i$. Therefore, $\mathcal{N}_i\neq \mathcal{N}_j$. \end{proof}

For each $i=1,2,\ldots,s-2$, it follows from Lemma \ref{cor-2} that
\[[\mathcal{N}_0\cap\mathcal{N}_1\cap\cdots\cap\mathcal{N}_{i-1}:\mathcal{N}_0\cap\mathcal{N}_1\cap\cdots\cap\mathcal{N}_{i}]=p.  \]
From this and Lemma \ref{lem-3}, 
\begin{equation}\label{eq-nova2}
   [\mathbb{Z}^s:\mathcal{N}_0\cap\cdots\cap\mathcal{N}_{i}]=p^i,
\end{equation}
for $i=1,2,\ldots,s-1$.

Returning to the modules $\mathcal{M}_p$ defined in Subsection \ref{subsection_Mp}, the previous equation leads to the following theorem:

\begin{theorem} \label{prop-3} Consider the $\mathbb{Z}$-modules $\mathcal{M}_p$ defined in (\ref{M_p_s_1}) for the case $s\equiv 1\pmod 2$ and in (\ref{M_p_s_0}) for the case $s\equiv 0\pmod 2$. Thus, $$[\mathcal{O}_{\mathbb{K}}:\mathcal{M}_p]=\left\{\begin{matrix}
   p^{\frac{s+1}{2}}  &  \mbox{if}~s\equiv 1\pmod 2,\\
   p^{\frac{s+2}{2}}  &  \mbox{if}~s\equiv 0\pmod 2. \end{matrix}\right.$$ 
\end{theorem}
\begin{proof}
As pointed in (\ref{eq_iso}), $\mathbb{Z}^s$ is a $\mathbb{Z}$-module isomorphic to $\mathcal{O}_\mathbb{K}$. By (\ref{eq-nova}), $\mathcal{M}_p$ is isomorphic to $\ker(\pi)$ as a $\mathbb{Z}$-module. Hence, $\mathbb{Z}^s/\ker(\pi)\simeq \mathcal{O}_\mathbb{K}/\mathcal{M}_p$. Therefore, $[\mathcal{O}_{\mathbb{K}}:\mathcal{M}_p]=[\mathbb{Z}^s:\ker(\pi)]$. The desired result now follows from (\ref{eq-nova2}).
\end{proof}

In both cases on $s$, we define a new $\mathbb{Z}$-submodule of $\mathcal{O}_\mathbb{K}$ given by
\begin{equation}\label{moduloMnew}
\tilde{\mathcal{M}}_p=\left\{{\bf x}\in\mathcal{M}_p:\displaystyle \sum_{i=0}^{s-1}a_i\equiv 0\pmod 2\right\}\subseteq\mathcal{M}_p.
\end{equation}



\begin{corollary}\label{corol_indice} The index of $\tilde{\mathcal{M}}_p$ in $\mathcal{O}_\mathbb{K}$ is given by
$$[\mathcal{O}_{\mathbb{K}}:\tilde{\mathcal{M}}_p]=\left\{\begin{matrix}
   2p^{\frac{s+1}{2}}  &  \mbox{if}~s\equiv 1\pmod 2,\\
   2p^{\frac{s+2}{2}}  &  \mbox{if}~s\equiv 0\pmod 2. \end{matrix}\right.$$ 
\end{corollary} \begin{proof} Let $\psi:\mathcal{M}_p\rightarrow \mathbb{Z}/2\mathbb{Z}$ be the surjective homomorphism of groups given by $\psi({\bf x})=\overline{\phi^{-1}({\bf x})\cdot{\bf \mu_0}}$ for each ${\bf x}\in\mathcal{M}_p$. Since $\ker(\psi) = \tilde{\mathcal{M}}_p$, we have that $\mathcal{M}_p/\tilde{\mathcal{M}}_p$ is isomorphic to $\mathbb{Z}/2\mathbb{Z}$, which implies $[\mathcal{M}_p:\tilde{\mathcal{M}}_p]=2$. Finally, the result follows from Theorem \ref{prop-3} and Lemma \ref{lem-3}.\end{proof}

\subsection{Center density bounds and examples} Firstly, consider $s\equiv 1\pmod 2$. By (\ref{eq-2}) and Corollary \ref{corol_indice}, the center density of the lattice obtained as image of the $\mathbb{Z}$-submodule $\tilde{\mathcal{M}}_p\subseteq \mathcal{M}_p\subseteq\mathcal{O}_\mathbb{K}$ via the canonical embedding is lower bounded by
\begin{equation}\label{conclusao_s1}
\delta(\sigma(\tilde{\mathcal{M}}_p)) \geq \frac{(2p^2)^{\frac{s}{2}}}{2^sp^{\frac{s-1}{2}}2p^{\frac{s+1}{2}}} = \frac{1}{2^{\frac{s+2}{2}}}.
\end{equation}

\begin{example} Consider $\mathbb{L}=\mathbb{Q}(\zeta_p)$ for some prime number $p$. If $s=3$ and $p \equiv 1 \pmod 3$,  \[ \delta(\sigma(\tilde{\mathcal{M}}_p)) = \frac{1}{4\sqrt{2}}. \]
Since the optimal center density in $\mathbb{R}^3$ is $\frac{1}{4\sqrt{2}}$, then the equality is attained in the inequality above. This means that $\sigma(\tilde{\mathcal{M}}_p)$ is the densest lattice in dimension 3. 
For example, consider $p=7$. In this case, $\mathbb{K} = \mathbb{Q}(t)$ is the maximal real subfield of $\mathbb{Q}(\zeta_7)$, where $t = \zeta_7 + \zeta_7^{-1}$. The canonical embedding $\sigma: \mathbb{K} \to \mathbb{R}^3$ maps the submodule $\tilde{\mathcal{M}}_p \subset \mathcal{O}_\mathbb{K}$ to the lattice $\sigma(\tilde{\mathcal{M}}_p) \subset \mathbb{R}^3$ with generator matrix $A$ given by
\[
A = \begin{pmatrix}
14 t_1 & 7(t_1 + t_2) & 2t_1 - 3t_2 + t_3 \\
14 t_2 & 7(t_2 + t_3) & 2t_2 - 3t_3 + t_1 \\
14 t_3 & 7(t_3 + t_1) & 2t_3 - 3t_1 + t_2
\end{pmatrix},
\]
where $t_1, t_2, t_3$ are the algebraic conjugates of $t$ (the roots of $x^3 + x^2 - 2x - 1 = 0$), explicitly given by
\[
t_1 = 2\cos\left(\frac{2\pi}{7}\right), \quad t_2 = 2\cos\left(\frac{4\pi}{7}\right), \quad \text{and} \quad t_3 = 2\cos\left(\frac{6\pi}{7}\right).
\]
Notice that $A = B \cdot C$, where $C$ is the coefficient matrix of the $\mathbb{Z}$-basis of $\tilde{\mathcal{M}}_p$ relative to $\{t, \theta(t), \theta^2(t)\}$, and $B$ is a generator matrix of $\mathcal{O}_K$:
\[
B = \begin{pmatrix}
t_1 & t_2 & t_3 \\
t_2 & t_3 & t_1 \\
t_3 & t_1 & t_2
\end{pmatrix}
\quad \text{and} \quad
C = \begin{pmatrix}
14 & 7 & 2 \\
0 & 7 & -3 \\
0 & 0 & 1
\end{pmatrix}.
\]
\end{example} 

\begin{example}
If $s=5$ and $p \equiv 1 \pmod 5$, it is possible to prove that the equality is also attained, \[ \delta(\sigma(\tilde{\mathcal{M}}_p)) = \frac{1}{8\sqrt{2}}, \] 
which is the center density of the best known lattice in dimension 5. In this case, $\mathbb{K} = \mathbb{Q}(\zeta_{11} + \zeta_{11}^{-1})$ is the maximal real subfield of the cyclotomic field $\mathbb{Q}(\zeta_{11})$. The Galois conjugates $t_1, t_2, t_3, t_4, t_5$ forming an integral basis of $\mathbb{K}$ over $\mathbb{Q}$ are given by
\[
t_i = 2 \cos\left( \frac{2^i \pi}{11} \right), \quad \text{for } i \in \{1, 2, 3, 4, 5\}.
\]
Under the canonical embedding $\sigma: K \to \mathbb{R}^5$, the generator matrix of the lattice $\sigma(\tilde{\mathcal{M}}_p)$ is given by the matrix product
\[
A = \begin{pmatrix}
t_1 & t_2 & t_3 & t_4 & t_5 \\
t_2 & t_3 & t_4 & t_5 & t_1 \\
t_3 & t_4 & t_5 & t_1 & t_2 \\
t_4 & t_5 & t_1 & t_2 & t_3 \\
t_5 & t_1 & t_2 & t_3 & t_4
\end{pmatrix}
\begin{pmatrix}
22 & 11 & 11 & 13 & 9 \\
0 & 11 & 0 & 7 & 6 \\
0 & 0 & 11 & 1 & 6 \\
0 & 0 & 0 & 1 & 0 \\
0 & 0 & 0 & 0 & 1
\end{pmatrix}
\]
\end{example} 



Given an odd integer $s$ and an odd prime $p$ satisfying $p\equiv 1 \pmod{s}$, there exists a unique cyclic extension $\mathbb{K}$ of $\mathbb{Q}$ of degree $s$ contained in $\mathbb{Q}(\zeta_p)$, by the Galois Correspondence. Consequently, the lattice $\sigma(\tilde{\mathcal{M}}_p)$ is well defined. Furthermore, by Dirichlet's Theorem on arithmetic progressions \cite[p.~253]{ribenboim}, there exist infinitely many odd primes $p$ satisfying $p\equiv 1 \pmod{s}$. Hence, for each odd dimension $s$, we obtain an infinite family of lattices $\{\sigma(\tilde{\mathcal{M}}_p)\}$ parametrized by the odd primes $p$ such that $p\equiv 1 \pmod{s}$.

Now consider $s\equiv 0\pmod{2}$. From Equation (\ref{eq-s-cong1}) and Corollary \ref{corol_indice}, \[ \delta(\sigma(\mathcal{M}_p)) \geq \frac{(p^2)^{\frac{s}{2}}}{2^sp^{\frac{s-1}{2}}p^{\frac{s+2}{2}}} = \frac{1}{2^s\sqrt{p}}. \] 
If $\frac{p-1}{s}$ is odd, then the trace form satisfies \[ Tr_{\mathbb{K}}({\bf x} \overline{\bf x}) = p\sum_{i=0}^{s-1}a_i^2-\frac{p-1}{s}\left(\sum_{i=0}^{s-1}a_i\right)^2 \equiv 0 \pmod{2}, \]
implying that its minimum is $\geq 2p^2$. Thus, \[ \delta(\sigma(\mathcal{M}_p)) \geq \frac{(2p^2)^{\frac{s}{2}}}{2^sp^{\frac{s-1}{2}}p^{\frac{s+2}{2}}} = \frac{1}{2^{s/2}\sqrt{p}}. \]

\begin{example}
    If $p=3$ and $s=2$, then \[ \delta(\sigma(\mathcal{M}_p)) \geq \frac{1}{2\sqrt{3}}.\] Since the optimal center density in $\mathbb{R}^2$ is $\frac{1}{2\sqrt{3}}$, which is provided by the hexagonal lattice $\Lambda_{hex}$, then $\sigma(\mathcal{M}_p)$ is the densest lattice in dimension $2$. In fact, in this case, $\mathbb{K} = \mathbb{Q}(\zeta_3)$ is the cyclotomic field of conductor $3$. The module $\mathcal{M}_p \subset \mathcal{O}_\mathbb{K}$ is given by the ideal $\mathcal{M}_p = 3\mathcal{O}_\mathbb{K}$. A generator matrix $A$ for the lattice $\Lambda = \sigma(\mathcal{M}_p)$ is given by
\[
A = \begin{pmatrix}
3 & -\dfrac{3}{2} \\[6pt]
0 & \dfrac{3\sqrt{3}}{2}
\end{pmatrix}.
\]
\end{example}

The examples above demonstrate that for specific low dimensions ($s = 2, 3, 5$), the constructed lattices attain the maximum or best known center densities in their respective dimensions. In general odd dimensions $s$, Dirichlet's Theorem on arithmetic progressions guarantees the existence of infinitely many primes $p \equiv 1 \pmod s$, yielding infinite parametric families of lattices $\{\sigma(\tilde{\mathcal{M}}_p)\}$ associated with each dimension $s$.

\section{Conclusion}

In this paper, we introduced a systematic approach to constructing algebraic lattices from subfields $\mathbb{K}$ of the cyclotomic field $\mathbb{Q}(\zeta_p)$. The primary theoretical contribution lies in defining the submodules $\mathcal{M}_p$ and $\tilde{\mathcal{M}}_p$ via orthogonality conditions over the finite field $\mathbb{F}_p$ and precisely computing their algebraic indices $[\mathcal{O}_{\mathbb{K}} : \mathcal{M}_p]$ and $[\mathcal{O}_{\mathbb{K}} : \tilde{\mathcal{M}}_p]$. By restricting the canonical trace form to these modules, we established lower bounds for their center densities in general dimensions $s$.  It is important to emphasize the scope and limitations of our results. While our algebraic construction recovers the densest lattice packings in dimensions 2 and 3, as well as the best known packing density in dimension 5, the bounds derived for general dimensions $s > 5$ serve primarily as theoretical lower bounds and construction methods.

The present work does not establish the optimality of these lattices in arbitrary dimensions, nor does it analyze the center density of the lattices $\sigma(\mathcal{M}_p)$ or $\sigma(\tilde{\mathcal{M}}_p)$ for $s \notin \{2, 3, 5\}$. These questions are left for future work. Moreover, future research may proceed in several promising directions, such as evaluating the performance of the generated lattice constellations over communication channels (e.g., Rayleigh fading and AWGN channels) by exploiting the algebraic diversity inherent to cyclotomic subfields, and extending this orthogonality-based module construction to subfields of cyclotomic fields $\mathbb{Q}(\zeta_m)$ where $m$ is a composite integer or a prime power $p^r$.

 \end{document}